%% file: Someidentities.tex
\documentclass[10pt,a4paper]{article}
\usepackage{amsmath,amsthm,amssymb,amsfonts}
\usepackage{hyperref}
\usepackage{mathtext}
\usepackage{enumitem}

\newcommand{\N}{{\mathbb{N}}}          
\newcommand{\R}{{\mathbb{R}}}          

\newcommand{\XIS}{{\mathfrak{X}}}

\newcommand{\SO}{{\mathrm{SO}}}

\newcommand{\rr}{\rightarrow}
\newcommand{\lrr}{\longrightarrow}

\newcommand{\call}{{\cal L}}             %
\newcommand{\calu}{{\cal U}}             %
\newcommand{\caly}{{\cal Y}}             %

\newcommand{\na}{{\nabla}}

\newcommand{\dx}{{\mathrm{d}}}

\newcommand{\inv}[1]{{#1}^{-1}}
\newcommand{\papa}[2]{\frac{\partial#1}{\partial#2}}

\newcommand{\estrela}{{\boldsymbol{\star}}}

\newcommand{\vol}{{\mathrm{vol}}}

\newtheorem{teo}{Theorem}[section]
\newtheorem{lemma}{Lemma}[section]
\newtheorem{coro}{Corollary}[section]
\newtheorem{prop}{Proposition}[section]

\def\cyclic{\mathop{\kern0.9ex{{+}
\kern-2.2ex\raise-.28ex\hbox{\Large\hbox{$\circlearrowright$}}}}\limits}

\title{Minkowski identities for hypersurfaces in constant sectional curvature manifolds}

\author{Rui Albuquerque}

\begin{document}


\maketitle


\begin{abstract}

We give a new proof of the generalized Minkowski identities relating the higher degree mean curvatures of orientable closed hypersurfaces immersed in a given constant sectional curvature manifold. Our methods rely on a \textit{fundamental differential system} of Riemannian geometry introduced by the author. We develop the notion of position vector field, which lies at the core of the Minkowski identities.

\end{abstract}


\ 
\vspace*{3mm}
\vspace*{2mm}

\vspace*{17mm}

\markright{\sl\hfill  R. Albuquerque \hfill}

\setcounter{section}{1}

\input{anpplicationofthefundESD.tex}

\vspace*{8mm}

\textsc{R. Albuquerque}\ \ \ \textbar\ \ \ 
{\texttt{rpa@uevora.pt}}

Departamento de Matem\'{a}tica da Universidade de \'{E}vora and
Centro de Investiga\c c\~ao em Mate\-m\'a\-ti\-ca e Aplica\c c\~oes

Rua Rom\~ao Ramalho, 59, 671-7000 \'Evora, Portugal

The research leading to these results has received funding from Funda\c c\~ao para a Ci\^encia e a Tecnologia.


\end{document}

%% file: anpplicationofthefundESD.tex
\begin{center}
 \textbf{1. Introduction}
\end{center}
\setcounter{section}{1}

The celebrated integral identities of Minkowski type (\cite{ChenYano,Hsiung1,Hsiung2,Katsurada}) read as follows.
Let $N$ be a closed orientable immersed hypersurface of Euclidean space and let $H_i$ denote the $i$th mean curvatures of $N$. Let ${P}$ be the position vector field from the origin. Then, for any $0\leq i\leq n-1$, we have
\begin{equation} \label{HMIzero}
 \int_N(H_i-\langle {P},\vec{n}\rangle H_{i+1})\,\vol_N=0.
\end{equation}
These identities were found by Hsiung and generalize the result of Minkowski for $i=0$.

The assertion that \eqref{HMIzero} is easy to prove using the exterior differential system of the ``$\theta,\alpha_0,\ldots,\alpha_n$'' is a challenge proposed to the reader of \cite{Alb2012}. The theory of what we regard as a fundamental differential system of Riemannian geometry is introduced in \cite{Alb2012}. We give here the solution to the challenge, in Theorem \ref{teo_HMI}, and continue to further develop the applications of the differential system.

In the first references above, the notion of the position vector field $P$ is quite obscure and the attempted generalization of the Minkowski identities to constant sectional curvature ambient $M$ yields a different result from what is found today in \cite{GuanLi} and \cite{GuidiMartino}. The latter recently discovered `Hsiung-Minkowski' identities are proved again in the present article. They appear either as a new formula or as a most important remark to add to Hsiung's assertions on space forms. The question surely depends on the notion of $P$. We propose here a definition and show that every warped product metric admits a position vector field.

We also prove in a global and invariant theory the generalized Minkowski identities for Killing vector fields, found by Katsurada in \cite{Katsurada}. These are vanishing theorems which are not so well-known today, perhaps due to the same difficulties pointed above. 

Our framework is that of the tangent sphere bundle of $M$, where the differential system lives. Though proves not enough to contain the identities known for any vector fields, rather than just the Killing. Killing vector fields do lift as extended vector fields to the sphere bundle. The theory must hence be extended to the whole Riemannian \textit{phase space} and its $\SO(n+1)$-structure.

We acknowledge the fruitful conversations with Vittorio Martino (U. Bologna) and Juan Sancho de Salas (U. Extremadura) regarding the clarification of some results. Also we thank the careful reading and remarks of the anonymous Referee of Differential Geometry and Applications.

Let us start by resuming with some definitions and notations from \cite{Alb2012}.

Throughout the text we let $(M,g)$ denote a smooth oriented Riemannian $(n+1)$-dimensional manifold. Quite often we use the notation $\langle\ ,\ \rangle$ for the metric.

We consider the unit tangent sphere bundle $\pi:SM\lrr M$ with its natural metric and $\SO(n)$ structure. This theory is surveyed in \cite{Alb2012,Alb2015a,Alb2019}. We denote by $e_0,e_1,\ldots,e_n,e_{n+1},\ldots,e_{2n}$ an adapted frame on $SM$, meaning the first $n+1$ vectors are horizontal and the remaining vectors are vertical, each one  a \textit{mirror} of the respective $e_1,\ldots,e_n$. Let us recall $\theta=e^0$ is the canonical contact 1-form due to Sasaki. $e_0$ is the restriction of the geodesic flow vector field $S$ to $SM$.

We have two useful ways of defining the $\alpha_i$ for each $i=0,\ldots,n$. Namely, letting $n_i=\frac{1}{i!(n-i)!}$ and $\alpha_n=e^{n+1}\wedge\cdots\wedge e^{n+n}$ be the volume-form of the fibers, we have
\begin{equation}
 \begin{split}
  \alpha_i  &= n_i\,\alpha_n\circ(B^{n-i}\wedge1^i)\\
  &  = n_i\sum_{\sigma\in S_n}\mathrm{sg}(\sigma)\,e^{\sigma_1}\wedge\cdots\wedge
e^{\sigma_{n-i}}\wedge e^{(n+\sigma_{n-i+1})}\wedge\cdots\wedge e^{(n+\sigma_n)} .
 \end{split}
\end{equation}
We also define $\alpha_{-1}=\alpha_{n+1}=0$. The mirror endomorphism $B$ of $TTM$ is well-defined, as it sends the horizontal lifts to the respective vertical and sends verticals to 0. The $\circ$ denotes an alternating operator.

We need to recall the canonical vertical vector field $\xi\in\XIS_{TM}$, defined by $\xi_u=u\in TTM$, $\forall u\in TM$ and the geodesic flow vector field on $TM$, the horizontal $S=B^\mathrm{t}\xi$.

Given an orientable hypersurface $f:N\hookrightarrow M$, let us recall the second fundamental form $A=\na\vec{n}$, where $\vec{n}$ is the unit-normal to $N$ with the induced orientation. Then the $i$th-mean curvature $H_i$ is defined by $\binom{n}{i}H_i$ being the elementary symmetric polynomial of degree $i$ on the eigenvalues $\lambda_1,\ldots,\lambda_n$ of $A$, the so-called principal curvatures of $N$. In other words, $\binom{n}{i}H_i=\sum_{1\leq j_1<\cdots<j_i\leq n}\lambda_{j_1}\cdots\lambda_{j_i}$. One also defines $H_0=1$.

On $N$ we have the canonical lift $\hat{f}:N\hookrightarrow SM$ of $f$ to $SM$, defined by $\hat{f}(x)=\vec{n}_{f(x)}$, and hence the formulas
\begin{equation} \label{fundamentalformulapullback}
 \hat{f}^*\theta=0
\end{equation}
and
\begin{equation}  \label{fundamentalformulapullba}
 \hat{f}^*{\alpha_i}=\binom{n}{i}H_i\,\vol_N .
\end{equation}

Next we let $M=\R^{n+1}$. Then the structural equations on the unit tangent sphere bundle of $M$ read, $\forall 0\leq i\leq n$,
\begin{equation}\label{derivadasalpha_i_flat}
\dx\alpha_i=(i+1)\,\theta\wedge\alpha_{i+1} .
\end{equation}

We stress the above alternating operator $\circ$ and mirror map $B$, the fundamental formulas \eqref{fundamentalformulapullback} and \eqref{fundamentalformulapullba} and this last structure equation \eqref{derivadasalpha_i_flat} are described, respectively, in Section 2, in Proposition 3.2 and in Example 3 of \cite{Alb2012}.
\begin{teo}[Hsiung-Minkowski identities] \label{teo_HMI}
Let $f:N\rr\R^{n+1}$ be a closed orientable immersed $C^2$ hypersurface of Euclidean space. Let ${P}$ be the position vector field. Then, for any $0\leq i\leq n-1$, we have
\begin{equation}
 \int_N(H_i-\langle {P},\vec{n}\rangle H_{i+1})\,\vol_N=0.
\end{equation}
\end{teo}
\begin{proof}
Let $(x,u)$ denote the natural coordinates of $S\R^{n+1}=\R^{n+1}\times S^n\subset T\R^{n+1}$. The position vector field $P$, defined by $P_x=x$, is lifted to a horizontal vector field $P_{(x,u)}=(x,0)$, clearly tangent to the sphere bundle. Since the 1-parameter subgroup of diffeomorphisms induced by $P$ on the base is given by $\phi_t(x)=e^tx$, the one induced on the total space of the bundle is given by
\[  \psi_t(x,u)=(e^tx,u) .  \]
Trivially we have that $B_{(x,u)}(\partial_{x^j})=\partial_{u^j}$ and $B_{(x,u)}(\partial_{u^j})=0$. Two easy computations with the two types of coordinate vector fields yield ${\psi_t}^*B=e^tB$. Hence the Lie derivative satisfies
\[  \call_{{P}}B=\left.\frac{\dx}{\dx t}\right\vert_0{\psi_t}^*B=B. \]
Now, since $\alpha_n=u^0\dx u^{1\cdots n}-u^1\dx u^{02\cdots n}+\cdots$ is independent of $x$, we have $\call_{P}\alpha_n=0$. By definition of the $\alpha_i$ and the basic technique with differential forms introduced in \cite[Section 4.1]{Alb2012}, on one hand we have a Leibniz rule
\begin{equation*}
\begin{split}
   \call_{P}\alpha_i& =n_i\,(\call_{P}\alpha_n)\circ
   (B^{n-i}\wedge1^i)+  n_i\,\alpha_n\circ
   (\call_{P}B\wedge B\wedge\cdots\wedge B\wedge1^i)+\cdots   \\ &   \hspace{45mm} \cdots +  n_i\,\alpha_n\circ
  (B\wedge\cdots\wedge B\wedge\call_{P}B\wedge1^i) \: 
  =\:(n-i)\,\alpha_i .
\end{split}
\end{equation*}
On the other hand, by Cartan's formula and \eqref{derivadasalpha_i_flat}, we find
\begin{align*}
   \call_{P}\alpha_i &= \dx(P\lrcorner\alpha_i)+P\lrcorner\dx\alpha_i \\    &= \dx(P\lrcorner\alpha_i)+ (i+1)\bigl(\theta(P)\,\alpha_{i+1}-\theta\wedge P\lrcorner\alpha_{i+1}\bigr) .
\end{align*}

Finally, let us bring $N$ into context. We have $\hat{f}(x)=(f(x),\vec{n}_{f(x)})$ and the formulas $\hat{f}^*\theta=0$ and $\hat{f}^*(\theta(P))_x=\langle P,S\rangle_{\hat{f}(x)}=\langle P,\vec{n}\rangle_{f(x)}$. By Stokes theorem, the two sides of the equations above yield
\[  (n-i)\int_N\binom{n}{i}H_i\,\vol_N = (i+1)\int_N\langle P,\vec{n}\rangle\binom{n}{i+1}H_{i+1}\,\vol_N ,  \]
redundant in case $i=n$. Finally we note that $(i+1)\binom{n}{i+1}=(n-i)\binom{n}{i}$.
\end{proof}


Case $i=0$ is due to Minkowski. The higher order formula is due to Hsiung, cf. \cite{Hsiung1}.

The same and several other Hsiung-Minkowski type identities for any closed submanifolds of Euclidean space are deduced in \cite{ChenYano,Katsurada,Rei}.

\vspace{2mm}
\begin{center}
 \textbf{2. Some general identities}
\end{center}
\setcounter{section}{2}

In the above theorem, if $N$ is a $C^2$ submanifold \textit{with} boundary, then a slightly more general identity of Hsiung-Minkowski is immediately found. Furthermore, the integral $\int_N\dx(X\lrcorner\alpha_i)$ appears often in similar settings below, so we continue to assume throughout that $N$ is closed.

The following lemma, besides its formulation, relies entirely on multilinear algebra; it does not owe to the differentiable structure of $\pi:SM\rr M$.
\begin{lemma}\label{formulacompostadealpha_nmenosumcomB}
The following identity is always satisfied:
 \begin{equation}
  \alpha_{n-1}\circ(B^{n-i}\wedge1^i)=i!(n-i+1)!\,\alpha_{i-1}.
 \end{equation}
\end{lemma}
\begin{proof}
Letting $B_1=\cdots=B_{n-i}=B$ and $B_{n-i+1}=\cdots=B_{n}=1$, we have an easy identity, $\alpha_i=n_i\,\alpha_n\circ(B^{n-i}\wedge1^{i})=n_i\sum_{\sigma\in S_n} 
e^{n+1}\circ B_{\sigma_1}\wedge\cdots\wedge e^{2n}\circ B_{\sigma_n}$, which follows immediately from definitions, cf. \cite[Section 4.3]{Alb2012}. In particular, we have that $\alpha_{n-1}=e^{1(n+2)(n+3)\cdots(2n)}+e^{(n+1)2(n+3)\cdots(2n)}+\cdots+e^{(n+1)(n+2)\cdots(2n-1)n}$. Here we notice the identity we wish to prove is true for $i=0$, because $e^j\circ B=0$, $\forall 0\leq j\leq n$, and $\alpha_{-1}=0$. We then proceed with the case $i>0$. First,
\begin{equation*}
 \begin{split}
  e^{(n+1)\cdots(n+j-1)j(n+j+1)\cdots(2n)}\circ (B^{n-i}\wedge1^i) =\qquad\qquad \\
 = \sum_{\sigma\in S_n} 
e^{n+1}\circ B_{\sigma_1}\wedge\cdots\wedge e^{n+j-1}\circ B_{\sigma_{j-1}}\wedge e^j\circ B_{\sigma_j} \wedge e^{n+j+1}\circ B_{\sigma_{j+1}} \wedge\cdots\wedge e^{2n}\circ B_{\sigma_n}  \\
 = \sum_{k=n-i+1}^n \sum_{\sigma\in S_n:\ \sigma_j=k} e^{n+1}\circ B_{\sigma_1}\wedge\cdots\wedge e^{n+j-1}\circ B_{\sigma_{j-1}}\wedge e^j\wedge e^{n+j+1}\circ B_{\sigma_{j+1}} \wedge\cdots\wedge e^{2n}\circ B_{\sigma_n}  \\
 = \sum_{k=n-i+1}^n \sum_{\sigma:\ \sigma_j=k} e^{n+1}\circ B_{\sigma_1}\wedge\cdots\wedge e^{n+j-1}\circ B_{\sigma_{j-1}}\wedge e^{n+j}\circ B\wedge e^{n+j+1}\circ B_{\sigma_{j+1}} \wedge\cdots\wedge e^{2n}\circ B_{\sigma_n} 
 \end{split}
\end{equation*}
since $e^j=e^{n+j}\circ B$. The sum is then equal to the following, where $B_1=\cdots=B_{n-i}=B$ and $B_{n-i+1}=\cdots\widehat{B_k}\cdots=B_{n}=1$, $B_k=B$,
\begin{equation*}
 \begin{split}
  = \sum_{k=n-i+1}^n \sum_{\sigma:\ \sigma_j=k} e^{n+1}\circ B_{\sigma_1}\wedge\cdots\wedge e^{n+j-1}\circ B_{\sigma_{j-1}}\wedge e^{n+j}\circ B_{\sigma_j}\wedge e^{n+j+1}\circ B_{\sigma_{j+1}} \wedge\cdots\wedge e^{2n}\circ B_{\sigma_n} .
 \end{split}
\end{equation*}
Then we may assume $B_1=\cdots=B_{n-i+1}=B$ and $B_{n-i+2}=\cdots=B_{n}=1$, so that the sum becomes
\begin{equation*}
 =i\sum_{\sigma:\ \sigma_j=n-i+1} e^{n+1}\circ B_{\sigma_1}\wedge\cdots\wedge e^{n+j-1}\circ B_{\sigma_{j-1}}\wedge e^{n+j}\circ B_{\sigma_j}\wedge e^{n+j+1}\circ B_{\sigma_{j+1}} \wedge\cdots\wedge e^{2n}\circ B_{\sigma_n} .
\end{equation*}
Returning to $\alpha_{n-1}=\sum_{j=1}^ne^{(n+1)\cdots(n+j-1)j(n+j+1)\cdots(2n)}$, we find
\begin{equation*}
 \begin{split}
  \alpha_{n-1}\circ(B^{n-i}\wedge1^i) &=i\sum_{j=1}^n\sum_{\sigma\in S_n:\ \sigma_j=n-i+1} e^{n+1}\circ B_{\sigma_1}\wedge\cdots\wedge e^{2n}\circ B_{\sigma_n} \\
  & =i(i-1)!(n-(i-1))!\alpha_{i-1} \\
  &=i!(n-i+1)! \alpha_{i-1},
 \end{split}
\end{equation*}
 the desired formula.
\end{proof}

In the following, $\na^*$ denotes the linear connection respecting the canonical splitting of $TTM\simeq\pi^*TM\oplus\pi^\estrela TM$, and therefore respecting the Sasaki metric. It is a \textit{double} pull-back to the tangent manifold of the Levi-Civita connection $\na$ of $M$, cf. \cite{Alb2012,Alb2019}. 

Recall the horizontal distribution is given by $\ker\na^*_\cdot\xi$ and that $\na^*_Y\xi=Y^v$ and, moreover, that $B$ is parallel for $\na^*$. The connection $\na^*$ is reducible to $\SO(n+1)$. Notice it does not comply with the further $\SO(n)$-reduction to the structure group of the submanifolds $TM\backslash0$ or $SM$.

\begin{prop}\label{Prop_derivadadeLiedeB}
 For any vector field $X$ over an open subset $\inv{\pi}(\calu)\subset TM$, for open $\calu\subset M$, and any $Y\in TTM$, we have
 \begin{equation}
  (\call_{X}B)Y=B\na^*_YX-\na^*_{BY}X.
 \end{equation}
 Moreover, $\call_{X^h}B=\lambda B$ for a horizontal vector field $X^h$ and some real function $\lambda$ on $\inv{\pi}(\calu)$ if and only if $\lambda$ is constant along the fibers and $X^h$ is the horizontal lift of a vector field $X$ on $\calu$ such that $\na_YX=\lambda Y$.
 \end{prop}
\begin{proof}
The torsion of the connection $\na^*$ has vertical part only; namely, the vertical lift $\pi^\estrela R(\ ,\ )\xi$ of the curvature of $M$ applied to $\xi$. Since both tensors $\pi^\estrela R(\ ,\ )\xi$ and $B$ vanish when applied on vertical directions, we find
 \begin{align*}
  (\call_{X}B)Y & =\call_{X}BY-B\call_{X}Y \\
  &=[X,BY]-B[X,Y]  \\
  &=\na^*_{X}BY-\na^*_{BY}X-B(\na^*_{X}Y-\na^*_{Y}X)   \\
  &= B\na^*_{Y}X-\na^*_{BY}X .
 \end{align*}
We remark for $X$ vertical this is just $-\na^*_{BY}X$. Now, if $X=X^h$ and $\call_{X^h}B=\lambda B$, then the horizontal part of the formula vanishes, $\na^*_{BY}X^h=0$, and so $X^h$ does not vary along the fibers. Thus $X^h=\pi^*X$ is a lift of $X\in\XIS_M$. Henceforth satisfying the identity $\na_YX=\lambda Y$ on $\calu$.
\end{proof}

The case where $X$ is the horizontal lift of the position vector field $P_x=x$ on the manifold $\R^{n+1}$, hence lifted as $P_x=(x,0)$, is quite interesting. Clearly, $\na^*_Y{P}=Y^h$, for any $Y\in TTM$. So the new formula yields immediately $\call_{P}B=B$, precisely as deduced for Theorem \ref{teo_HMI}. 

We give the name of \textit{$\lambda$-mirror} to the vector fields on $TM$ satisfying $\call_{X}B=\lambda B$, for some scalar $\lambda$. For constant $\lambda$, they are all found as $\lambda X_1$, with $X_1$ a particular solution of the 1-mirror equation, plus the Lie algebra of 0-mirror vector fields. Indeed a Lie algebra.

The previous statements are verified mutatis mutandis with the notion of \textit{$\lambda$-adjoint-mirror} vector field on $TM$, ie. those $X$ such that $\call_{X}B^\mathrm{t}=\lambda B^\mathrm{t}$.

Other details on mirror vector fields can be found in \cite{Alb2019}.

\vspace{2mm}
\begin{center}
 \textbf{3. On constant sectional curvature}
\end{center}
\setcounter{section}{3}

Let us assume that $M$ has constant sectional curvature $c$. From \cite[Example 3, Section 2.4]{Alb2012}, cf. \eqref{derivadasalpha_i_flat}, we have the \textit{magic} formula
\begin{equation}\label{derivadasalpha_iemCSecC}
 \dx\alpha_i=\theta\wedge\bigl((i+1)\alpha_{i+1}-c(n-i+1)\alpha_{i-1}\bigr) .
\end{equation}

Before considering a particular vector field arising from $M$ we may study the geodesic spray $S=\theta^\sharp=B^\mathrm{t}\xi$. We find the following results.
\begin{prop}
Let $M$ have constant sectional curvature $c$. Then:
 \begin{enumerate}[label=(\roman*)]
  \item \:$\call_SB=1-2B^\mathrm{t}B$
  \item \:$\call_S\alpha_i=(i+1)\,\alpha_{i+1}-c(n-i+1)\,\alpha_{i-1}$.
 \end{enumerate}
\end{prop}
\begin{proof}
 (i) The canonical vector field $\xi$ induces a projection of $TTM$ onto the vertical tangent subbundle $\ker\dx\pi$, through $\na^*_Y\xi=Y^v,\ \forall Y\in TTM$, with kernel the horizontal tangent subbundle. Therefore 
 $(\call_SB)Y=B\na^*_YS-\na^*_{BY}S=Y^v-Y^h$, $\forall Y\in TTM$.\\
 (ii) Notice $S\lrcorner\alpha_i=0,\ \forall 0\leq i\leq n$ for there is no $\theta$ factor. Using Cartan formula we deduce $\call_S\alpha_n=S\lrcorner\dx\alpha_n=-c\,\alpha_{n-1}$. Of course the same argument, together with \eqref{derivadasalpha_iemCSecC}, applies to any $\alpha_i$, proving immediately the desired formula. We wish to take a different path, somehow strength testing the exterior differential system. First, by an already seen Leibniz rule, we have
 \[  \call_S\alpha_i\:=\:n_i\,(\call_S\alpha_n)\circ (B^{n-i}\wedge1^i)+n_i\,\alpha_n\circ (\call_SB^{n-i}\wedge1^i) . \]
Using Lemma \ref{formulacompostadealpha_nmenosumcomB}, the first term is
 \begin{align*}
  n_i\,(\call_S\alpha_n)\circ (B^{n-i}\wedge1^i)&=
 -cn_i\,\alpha_{n-1}\circ(B^{n-i}\wedge1^i)  \\
 &=-cn_ii!(n-i+1)!\,\alpha_{i-1} \\
 &=-c(n-i+1)\,\alpha_{i-1}.
  \end{align*}
 Regarding the second term, with any $k\in\N$ we have
 \[   \call_SB^k=(1-2B^\mathrm{t}B)\wedge B^{k-1}+B\wedge(1-2B^\mathrm{t}B)\wedge B^{k-2}+\cdots  \]
and, since $e^{j+n}\circ B^\mathrm{t}B=0$, $\forall j=1,\ldots,n$, we find
\begin{align*}
 \lefteqn{ n_i\,\alpha_n\circ(\call_SB^{n-i}\wedge1^i)\:= }
   \ \ \ \ \ \ \  \\  
   &=  n_i\,\alpha_n \circ (1\wedge B^{n-i-1}\wedge1^i)+  
    n_i\,\alpha_n\circ(B\wedge 1\wedge B^{n-i-2}\wedge1^i)+\cdots \\
 & =\frac{1}{i!(n-i)!}(n-i)\,\alpha_n\circ(B^{n-(i+1)}\wedge1^{i+1})\\
     &=(i+1)\,\alpha_{i+1}
\end{align*}
as we wished.
\end{proof}
The last result gives a work-around to obtain a complicated basic structural equation:
\begin{equation}
 \alpha_j\wedge\alpha_{n-j}=
  (-1)^j\binom{n}{j}\,\alpha_0\wedge\alpha_n , 
  \qquad\forall 0\leq j\leq n .
\end{equation}

The following results are independent of the above introduction.

We shall see that a constant sectional curvature $c$ manifold $M$ admitting a horizontal constant $\lambda$-mirror vector field must be flat.

In order to generalize the Hsiung-Minkowski identities to curved spaces we require a position vector field. In order to prove them, we need the $\lambda$-mirror condition.

A position vector field $P$ exists on $M$ and yields a horizontal $\lambda_{c}$-mirror vector field $\pi^*P$. We refer to  \cite{GuanLi,GuidiMartino,MartinoTralli} for the previous assertion, where we see that $P$ in polar coordinates is given by $P=s_c(r)\partial_r$, with $r$ the geodesic distance function from a base point in $M$. Such vector field $P$ satisfies, $\forall Y\in TM$,
\begin{equation}
 \na_YP=\lambda_cY,
\end{equation}
with $s_c,\ \lambda_{c}=s_c'$ the functions
\begin{equation}
   s_c(r)=\begin{cases}
           \frac{1}{\sqrt{c}}\sin(\sqrt{c}r), & c>0  \\
           r, & c=0 \\
           \frac{1}{\sqrt{-c}}\sinh(\sqrt{-c}r), & c<0            
          \end{cases},\qquad
    \lambda_c(r)=\begin{cases}
           \cos(\sqrt{c}r),  & c>0  \\
           1, & c=0 \\
           \cosh(\sqrt{-c}r), & c<0           
           \end{cases}.
\end{equation}

In Corollary \ref{Coro_constructionofPositionvectorfields} below we generalize the construction of local position vector fields.

The next theorem is found in the works of P. Guan and J. Li \cite{GuanLi} and C. Guidi and V. Martino \cite{GuidiMartino}. In their proofs, the authors recur to Newton's identities for symmetric polynomials, which play a central role just like they played originally in Reilly's proof, in \cite{Rei}, of the Hsiung-Minkowski identities for Euclidean space.

\begin{teo}[cf. \cite{GuanLi,GuidiMartino}] \label{teo_GLGMidentities}
 Let $M$ be an oriented $(n+1)$-dimensional manifold of constant sectional curvature $c$. Let $P$ be a position vector field ($\na_YP=\lambda_{c}Y$) defined on a neighborhood of a given closed oriented immersed hypersurface $f:N\rr M$. Then, for any $0\leq i\leq n-1$, we have 
\begin{equation}  \label{HsiungMinkowskiKatsuradadetc}
 \int_N(\lambda_{c} H_i-\langle {P},\vec{n}\rangle H_{i+1})\,\vol_N=0.
\end{equation}       
\end{teo}
\begin{proof}
 The function $\lambda_{c}$ on $M$ gives rise to \textit{another} function on $TM$ such that $\na^*_{\pi^*Y}\pi^*P=\lambda_{c}\pi^*Y$, for all $Y\in TM$. Recurring to Proposition \ref{Prop_derivadadeLiedeB} it follows easily that $\call_{\pi^*P}B=\lambda_{c}B$. Let us denote \textit{also} by $P=\pi^*P$ the horizontal lift of $P$. In the following, notice $P\lrcorner\alpha_n=0$ and recall $\alpha_{n+1}=0$ by definition. Then
 \begin{align*}
  \call_{P}\alpha_i &= 
  n_i\,(\call_{P}\alpha_n)\circ(B^{n-i}\wedge1^i) + n_i\,\alpha_n\circ(\call_{P}B^{n-i}\wedge1^i)  \\ 
  &= n_i\,({P}\lrcorner\dx\alpha_n)\circ(B^{n-i}\wedge1^i)+ (n-i)\lambda_{c} n_i\,\alpha_n\circ(B^{n-i}\wedge1^i) \\
  &= n_i\bigl(-c\theta(P)\,\alpha_{n-1}+c\,\theta\wedge P\lrcorner\alpha_{n-1}\bigr)\circ(B^{n-i}\wedge1^i) + (n-i)\lambda_{c}\,\alpha_i .
 \end{align*}
 With the induced map $\hat{f}:N\rr SM$, since $\theta\circ B=0$ and $\hat{f}^*\theta=0$, we have again by Lemma \ref{formulacompostadealpha_nmenosumcomB}
 \begin{equation*}
 \begin{split}
   \hat{f}^*\call_{P}\alpha_i  & =
  -n_ic\hat{f}^*\bigl(\theta(P)\,\alpha_{n-1}\circ(B^{n-i}\wedge1^i)\bigr) + (n-i)\lambda_{c}\,\hat{f}^*\alpha_i \\
  &= -c(n-i+1)\langle P,\vec{n}\rangle\,\hat{f}^*\alpha_{i-1} + (n-i)\lambda_{c}\,\hat{f}^*\alpha_i. 
 \end{split} 
 \end{equation*}
 On the other hand, by Cartan's formula, $\call_{P}\alpha_i=\dx(P\lrcorner\alpha_i)+P\lrcorner\dx\alpha_i$. Then we recall  \eqref{derivadasalpha_iemCSecC} and integrate on the hypersurface with empty boundary $N$:
 \begin{align}
  \int_N\hat{f}^*\call_{P}\alpha_i &=
     \int_N\hat{f}^*(P\lrcorner\dx\alpha_i) \nonumber\\ 
     &=  \int_N\hat{f}^*\bigl(P\lrcorner((i+1)\,\theta\wedge\alpha_{i+1}-c(n-i+1)\,\theta\wedge\alpha_{i-1})\bigr)  \nonumber \\
     &= \int_N\langle P,\vec{n}\rangle\hat{f}^*\bigl((i+1)\,\alpha_{i+1}-c(n-i+1)\,\alpha_{i-1}\bigr).   \label{formulaReilly}
 \end{align}
 Comparing with the above me deduce $(n-i)\int_N\lambda_{c}\hat{f}^*\alpha_i=(i+1)\int_N\langle P,\vec{n}\rangle\hat{f}^*\alpha_{i+1}$, $\forall 0\leq i\leq n$. The result follows, for all $i<n$, just as in the Euclidean case.
\end{proof}
Case $c=0$ gives again the Hsiung-Minkowski identities.

Our proof clearly depends on the intrinsic geometry of $M$ and $SM$. The magic formula \eqref{derivadasalpha_iemCSecC} plays the role of Newton's identities in other proofs, yet the latter are required in such proofs \textit{after} the hypersurface appears in context.

When we replace \eqref{fundamentalformulapullback} in \eqref{formulaReilly}, as remarked in \cite{Alb2012}, we see our formula is similar to a variational calculus derivative of the $i$th-mean curvature functional $N\rightsquigarrow\binom{n}{i}\int_NH_i\,\vol_N$ first found by R.C.~Reilly.

We remark that taking vertical lifts in the above proof, instead of $\pi^*P$, seems worthless; for all sides vanish identically.

Since any constant vector field $v_0$ is parallel on Euclidean space, one finds a very particular case of a formula of Katsurada, cf. Theorem \ref{Katsurada} below: for every closed oriented immersed hypersurface $N\subset\R^{n+1}$, $\forall j=1,\ldots,n$, $\forall v_0\in\R^{n+1}$, we have
\begin{equation}\label{forparallel}
    \int_N\langle v_0,\vec{n}\rangle H_j\,\vol_N =0 .  
\end{equation}
The proof of this formula is straightforward as the above. Of course, other constant curvature spaces do not admit parallel vector fields. 

What we wish to observe is that the Hsiung-Minkowski identities, as expected, are invariant too of any base point and position vector field $P_x=x-v_0$.

\vspace{2mm}
\begin{center}
 \textbf{4. Position vector fields}
\end{center}
\setcounter{section}{4}

Here we prove that warped product metrics admit a position vector field, this is, a vector field $P$ on $M$ such that $\na_XP=\lambda X$, $\forall X\in\XIS_M$, for some function $\lambda$, cf. Theorem \ref{teo_GLGMidentities}.

On any Riemannian manifold $M,g$, a position vector field is conformal-Killing: $\call_Pg=2\lambda g$.

Regarding the tangent manifold we may draw the following conclusion.
\begin{prop}\label{Prop_horizontalmirrorvectorfield}
 There exists a horizontal $\lambda$-mirror vector field $P$ if and only if $P$ is the lift of a position vector field $P$ on $M$. In this case, $R(X,Y)P=\dx\lambda(X)Y-\dx\lambda(Y)X$ and, if $\lambda$ is a constant, then every plane containing $P$ is flat.
\end{prop}
\begin{proof}
 The first part combines previous definitions with Proposition \ref{Prop_derivadadeLiedeB}: $P$ is a horizontal lift and on the base $M$ we have $\na_YP=\lambda Y$, for all $Y\in TM$. Then we find $R(X,Y)P=\na_X\na_YP-\na_Y\na_XP-\na_{[X,Y]}P=\na_X\lambda Y-\na_Y\lambda X-\lambda[X,Y]=\dx\lambda(X)Y-\dx\lambda(Y)X$.
\end{proof}

Any metric $g$ on an $(n+1)$-Riemannian manifold $M$ may be locally written as $\R^+\times S^n$ with
\begin{equation}
 g=\dx r^2+g_{S^n(r)} .
\end{equation}
Indeed, each geodesic ray or, equivalently, each exponential map line with the direction of $\partial_r$, is orthogonal to the sphere $\{r\}\times S^n=S^n(r)$, $\forall r$. Such is the conclusion of the well-known Gauss Lemma.

We now consider a milder situation, where $M=\R\times\caly$ with $(\caly,g_\caly)$ a given Riemannian manifold. We further assume there is a positive function $\psi=\psi(r,y)$ on $M$ such that the metric on $M$ is
\begin{equation}
 g=\dx r^2+(\psi(r,y))^2g_{\caly} .
\end{equation}

We may then describe the Levi-Civita connection in terms of the Levi-Civita connection $\na^{(r)}$ of $\caly(r)$. Let us call $X\in\XIS_M$ a \textit{horizontal lift} if $X$ is tangent to $\caly$ and does not depend of $r$.
\begin{teo}
 The Levi-Civita connection $\na$ of $g$ satisfies:
\begin{equation}
  \na_{\partial_r}\partial_r=0,\qquad\na_{\partial_r}Y=\na_Y\partial_r=fY,\qquad
  \na_ZY=\na^{(r)}_ZY-fg(Y,Z)\partial_r
\end{equation}
where $Y,Z$ are horizontal lifts and $f=f(r,y)$ is a function such that
\begin{equation}
 \papa{\psi}{r}-f\psi=0 .
\end{equation}
\end{teo}
\begin{proof}
  One may restrict to $\partial_r$ and to $Y,Z$ horizontal lifts in the analysis of the torsion equation, which is trivial, and in the six cases verification of $\na g=0$. 
\end{proof}
Recall warped product metrics are defined as above with $\psi$ just a function of $r$. Hence $Y(\psi)=0,\ Y\in T\caly$ and the following application takes place.
\begin{coro} \label{Coro_constructionofPositionvectorfields}
 Let $P$ denote the vector field $P=\psi\partial_r$. Then 
 \begin{equation}
  \na_{\partial_r}P=\papa{\psi}{r}\partial_r\qquad \mbox{and}\qquad \na_YP=Y(\psi)\partial_r+\papa{\psi}{r}Y.
 \end{equation}
 If $g$ is a warped product metric, then $P$ is a position vector field.
\end{coro}

Example: Let us see the last result in a different perspective. Let $M=M(c)=S^{n+1}(R_0)\subset\R^{n+1+1}$ denote the sphere of radius $R_0$ with coordinates $(x,t)$. As it is well known, $c=\frac{1}{R_0^2}$. Let $N_t\subset M(c)$ denote the $n$-spheres $\|x\|^2=R_0^2-t^2=:R_t^2$ along the obvious axis, for each $t\in]-R_0,R_0[$. Then $\vec{n}_{(x,t)}=\frac{1}{R_0R_t}(-tx,R_t^2)$ is a unit normal of $N_t$ in $M(c)$; of course, $TN_t=\{Y=(Y,0):\ Y\perp x\}$. Notice $\vec{N}=\frac{1}{R_0}(x,t)$ is the unit normal to $M(c)$ at $(x,t)\in N_t$. We have $\dx\vec{n}(Y)=-\frac{t}{R_0R_t}Y$, also because $\dx(\|x\|)(Y)=0$. Then $\dx\vec{n}(Y)\perp \vec{N}$ and therefore $\na^{M}_{Y}\vec{n}=-\frac{t}{R_0R_t}Y$. We conclude the hypersurfaces $N_t$ are umbilic, with all $\lambda_i=-\frac{t}{R_0R_t}:=\lambda$. Henceforth $H_i=(-1)^i\frac{t^i}{R_0^iR_t^i}$, for all $0\leq i\leq n$.

We also find $\partial_t\vec{n}-\langle\partial_t\vec{n},\vec{N}\rangle\vec{N}=\frac{t^3}{R_0^2R_t^2}\vec{n}$ and  $\dx\vec{n}(x)-\langle\dx\vec{n}(x),\vec{N}\rangle\vec{N}=\frac{t^2}{R_0^2}\vec{n}$. It follows that $\na^M_{\vec{n}}\vec{n}=0$. Next we search for a position vector field of the kind $P=a\vec{n}$ with $a$ function of $t$; it will be sufficient to find that $\na^M_{\vec{n}}P=b\vec{n}$ and $\na^M_YP=bY:=-\frac{ta}{R_0R_t}Y$, because  $\dx t(Y)=0$ for all $Y\perp x$. Immediately we verify the space form Minkowski identity $bH_i-\langle P,\vec{n}\rangle H_{i+1}=0$ over the hypersurface $N_t$. Now, another computation yields $\dx t(x)=-\frac{R_t^2}{t}$ and thus $\na^M_{\vec{n}}P=\dx a(\vec{n})\vec{n}+a\na^M_{\vec{n}}\vec{n}=2\papa{a}{t}\frac{R_t}{R_0}\vec{n}$. Hence, we must have  $b=-\frac{ta}{R_0R_t}=2\papa{a}{t}\frac{R_t}{R_0}$. Up to a constant factor, $a=\sqrt{R_t}=\sqrt[4]{R_0^2-t^2}$.

\vspace{2mm}
\begin{center}
 \textbf{5. Extended vector fields and Katsurada identities}
\end{center}
\setcounter{section}{5}

In the celebrated article on the geometry of tangent bundles, Sasaki introduces along with \textit{his} famous metric the notion of \textit{extended} vector field $\widetilde{X}\in\XIS_{TM}$ of any given vector field $X\in\XIS_M$, cf. \cite{Alb2019} and the references therein. It is a definition of the most natural kind, not requiring any metric or any connection defined on the given manifold. The extended vector field $\widetilde{X}$ is also known as the \textit{complete lift} of $X$. 

In a coordinate chart $(x^1,\ldots,x^{n+1})$ of $M$, giving\footnote{N.B.: One may consistently use the notation:\ \:$ \partial_j=\partial_{x^j},\:\ \pi^*\partial_j=\partial_j-v^i\Gamma_{ji}^k\partial_{v^k},\ \:\pi^\estrela\partial_k=\partial_{v^k}$.} the coordinates $(x^j,v^j)$ on $TM$, we have $X=X^j\partial_j$ and $\widetilde{X}=X^j\partial_j+v^j\papa{X^k}{x^j}\partial_{v^k}$. This shows the following formula is independent of the torsion free connection (recall the notation for horizontal and vertical lifts):
\begin{equation}
 \widetilde{X}=\pi^*X+\na^*_S\pi^\estrela X.
\end{equation}

It follows that $\call_{\widetilde{X}}B=0$, cf. \cite[Proposition 3.2]{Alb2019}. This is immediate by Proposition \ref{Prop_derivadadeLiedeB} and by recalling $S=v^j\pi^*\partial_j=v^j\partial_j-v^jv^i\Gamma_{ij}^k\partial_{v^k},\ \xi=v^k\partial_{v^k}$ and $\na^*_Y\xi=Y^v$ and $\na^*_{BY}\na^*_S\pi^\estrela X=\na^*_Y\pi^\estrela X$.

We give here a more elementary proof. In truth, the mirror map $B$ satisfies the identity $B=\dx x^j\otimes\partial_{v^j}$. In other words, $B$ does not depend on the connection either (fact which is not true for $B^\mathrm{t}$). Therefore
\begin{eqnarray}
 \call_{\widetilde{X}}(\dx x^j\otimes\partial_{v^j})=
 \dx({\widetilde{X}}\lrcorner\dx x^j)\otimes\partial_{v^j}+\dx x^j\otimes[{\widetilde{X}},\partial_{v^j}]=\papa{X^j}{x^k}\dx x^k\otimes\partial_{v^j}-\dx x^j\otimes\papa{X^k}{x^j}\partial_{v^k} =0.
\end{eqnarray}

Let us recall the torsion of $\na^*$ on $TM$. We have $\na^*_XY-\na^*_YX-[X,Y]=\pi^\estrela R(X,Y)\xi$. For instance, $[S,\xi]=\na^*_S\xi-\na^*_\xi S=-B^\mathrm{t}\na^*_\xi\xi=-B^\mathrm{t}\xi=-S$. This result is also easily deduced with coordinates. 

Additional properties are found with Killing vector fields.
\begin{prop}[Sasaki]
Suppose $X$ is a Killing vector field on $M$. Then $\widetilde{X}$ is Killing for the Sasaki metric on $TM$. Moreover, $\widetilde{X}$ is tangent to $SM$ and Killing for the induced metric.
\end{prop}
\begin{proof}
 The deduction of the first part can be seen in \cite{Alb2019}: infinitesimal isometries lift as infinitesimal isometries. The fact that $\widetilde{X}$ is tangent to $SM$ is also due to Sasaki, but we provide an immediate proof. Recall $SM$ is the locus of $\|\xi\|^2=1$ and thus $TSM=\xi^\perp$; then, regarding the vertical side of $\widetilde{X}$, we find $\langle\na^*_S\pi^\estrela X,\xi\rangle=\langle\na^*_S\pi^*X,S\rangle=0$ by skew-symmetry.
\end{proof}
We may now legitimately use $\call_{\widetilde{X}}$ over the submanifold $SM$, for sections of any vector bundle of tensors \textit{a priori} defined over $TM$.
\begin{teo}
Let $X\in\XIS_M$ be Killing. Then:
 \begin{enumerate}[label=(\roman*)]
  \item $\call_{\widetilde{X}}B^\mathrm{t}=0$ and $\call_{\widetilde{X}}B=0$
  \item $\call_{\widetilde{X}}\pi^*\vol=0$ and $\call_{\widetilde{X}}\pi^\estrela\vol=0$
  \item $\call_{\widetilde{X}}S=[\widetilde{X},S]=0$ and $\call_{\widetilde{X}}\xi=[\widetilde{X},\xi]=0$
  \item $\call_{\widetilde{X}}\alpha_i=0$, $\forall 0\leq i\leq n$.
 \end{enumerate}
\end{teo}
\begin{proof}
 (i) Since $X$ is an infinitesimal affine transformation, the 0-adjoint-mirror equation follows from \cite[Proposition 3.5]{Alb2019}. Now  0-adjoint-mirror implies 0-mirror by \cite[Proposition 3.4]{Alb2019}. If one prefers, as seen above, we have always $\call_{\widetilde{X}}B=0$.\\
 (ii) Since $X$ is divergent free, $\call_{\widetilde{X}}\pi^*\vol=\dx({\widetilde{X}}\lrcorner\pi^*\vol)=\dx(\pi^*X\lrcorner\pi^*\vol)=\pi^*\call_{X}\vol=0$. Then we have
  \[  \pi^\estrela\vol=\frac{1}{n!}\pi^*\vol\circ(B^\mathrm{t}\wedge\cdots\wedge B^\mathrm{t})  \]
  and the result follows.\\
 (iii) $[\widetilde{X},S]=\na^*_{\widetilde{X}}S-\na^*_S\widetilde{X}-\pi^\estrela R(\widetilde{X},S)\xi=B^\mathrm{t}(\na^*_S\pi^\estrela X)-\na^*_S\pi^*X-\na^*_S\na^*_S\pi^\estrela X-\pi^\estrela R(\pi^*X,S)\xi=-\na^*_S\na^*_S\pi^\estrela X-\pi^\estrela (R(X,\pi_*S)\pi_*S)=-\pi^\estrela(\na^2X(\pi_*S,\pi_*S)+R(X,\pi_*S)\pi_*S)=0$, following from the equation of $X$ being an infinitesimal affine transformation. Finally $[{\widetilde{X}},\xi]=\na^*_{\widetilde{X}}\xi-\na^*_\xi\widetilde{X}-\pi^\estrela R(\widetilde{X},\xi)\xi=\na^*_S\pi^\estrela X-\na^*_S\pi^\estrela X=0$.\\
(iv) The result can be checked in two different ways, both requiring $\call_{\widetilde{X}}B^\mathrm{t}=0$. Recall the formula for the $\alpha_i=\frac{1}{i!(n-i)!}\alpha_n\circ(B^{n-i}\wedge1^i)$; thus essentially we have to compute $\call_{\widetilde{X}}\alpha_n$. We have
\[  \alpha_n=\xi\lrcorner\pi^\estrela\vol=\frac{1}{n!}(S\lrcorner\pi^*\vol)\circ(B^\mathrm{t}\wedge\cdots\wedge B^\mathrm{t}) . \]
Since $\call_{\widetilde{X}}(S\lrcorner\pi^*\vol)=(\call_{\widetilde{X}}S)\lrcorner\pi^*\vol+S\lrcorner\call_{\widetilde{X}}\pi^*\vol$,
the result follows. Notice $S\lrcorner\pi^*\vol=\alpha_0$, so we could equally focus on $\alpha_i=n_i\,\alpha_0\circ(1^{n-i}\wedge {B^\mathrm{t}}^i)$.
\end{proof}
Finally, let us resume with $(M,g)$ an $(n+1)$-dimensional Riemannian manifold of constant sectional curvature $c$. 

Let $X$ be a Killing vector field on the neighborhood of an oriented closed immersed hypersurface $f:N\rr M$. Then, for all $i=0,\ldots n$,
 \begin{equation}
  \int_N\langle X,\vec{n}\rangle((i+1){\hat{f}}^*\alpha_{i+1}-c(n-
  i+1){\hat{f}}^*\alpha_{i-1})=0.
 \end{equation}
Indeed, applying formula \eqref{derivadasalpha_iemCSecC} we have $
  0=\call_{\widetilde{X}}\alpha_i =\dx({\widetilde{X}}\lrcorner\alpha_i)+(\theta({\widetilde{X}})- \theta\wedge {\widetilde{X}}\lrcorner)\bigl((i+1)\alpha_{i+1}-c(n-i+1)\alpha_{i-1}\bigr)$
and the result follows. Now recalling $\alpha_{-1}=\alpha_{n+1}=0$, we obtain the following particular case of the generalized Minkowski identities of Katsurada, partly by induction.
\begin{teo}[Katsurada identities]   \label{Katsurada}
In the previous conditions,
 \begin{enumerate}[label=(\roman*)]
  \item If $c=0$ or $n$ is odd, then all $\int_N\langle X,\vec{n}\rangle H_j=0$.
  \item If $n$ is even, then all odd $\int_N\langle X,\vec{n}\rangle H_{2j+1}=0$ and
  \begin{equation}\label{KatsuradaIdent_for_induction}
   \int_N\langle X,\vec{n}\rangle H_{i+1}\,\vol=\frac{ic}{n-i}\int_N\langle X,\vec{n}\rangle H_{i-1}\,\vol.
  \end{equation}
 \end{enumerate}
\end{teo}
Katsurada's result, \cite[Formulas (I)$_{\mathrm{i}}$ and (II)$_{\mathrm{i}}$]{Katsurada}, found for a Killing vector field on a hypersurface of a constant sectional curvature manifold $M$, goes farther: it makes no restrictions, neither in order or in dimension, and asserts all integrals vanish identically. Our methods have not yield to such generality.

No counter-example is known yet of non-vanishing \eqref{KatsuradaIdent_for_induction}.

%% file: Someidentities.bbl
\vspace*{10mm}







\begin{thebibliography}{30}




\bibitem{Alb2012}
R. Albuquerque,
\emph{A fundamental differential system of Riemannian geometry}, to appear in Revista Iberoamericana de Matem\'atica, Issue 36.\  (2020), \url{http://dx.doi.org/10.4171/rmi/1118}.


\bibitem{Alb2015a}
---,
\emph{A fundamental differential system of 3-dimensional Riemannian geometry},
Bull. Sci. Math. 143 (2018), 82--107.


\bibitem{Alb2019}
---,
{\em Notes on the Sasaki metric}, Expo. Math. 37 (2019), 207--224. \url{https://dx.doi.org/10.1016/j.exmath.2018.10.005}.




\bibitem{ChenYano}
B.-Y. Chen and K. Yano,
{\em Integral formulas for submanifolds and their applications},
J. Differ. Geom. 5(3-4) (1971), 467--477.


\bibitem{GuanLi}
P. Guan and J. Li,
{\em A mean curvature type flow in space forms}.
Int. Math. Res. Not. (2015), 4716--4740.


\bibitem{GuidiMartino}
C. Guidi and V. Martino,
{\em Horizontal Newton operators and high-order Minkowski formula}, submitted. \url{https://www.dm.unibo.it/~martino/research.html}


\bibitem{Hsiung1}
C.-C. Hsiung,
\emph{Some integral formulas for closed hypersurfaces},
Math. Scand. 2 (1954), 286--294.



\bibitem{Hsiung2}
C.-C. Hsiung,
\emph{Some integral formulas for closed hypersurfaces in Riemannian space},
Pacific J. Math. 6 (1956), 291--299.


\bibitem{Katsurada}
 Y. Katsurada,
\emph{Generalized Minkowski formulas for closed hypersurfaces in Riemann space},
Annali Mate. Pura ed Appl. (4) 57 (1962), 283--293.




\bibitem{MartinoTralli}
V. Martino and G. Tralli,
{\em On the Minkowski formula for hypersurfaces in complex space forms}, submitted. \url{https://www.dm.unibo.it/~martino/research.html}


\bibitem{Rei}
R.~C.~Reilly,
\emph{Variational properties of functions of the mean curvatures for hypersurfaces in space forms},
J. Differential Geometry 8 (1973), 465--477.




\end{thebibliography}
